\documentclass{amsart}

\usepackage{amsmath,amsfonts,amssymb,stmaryrd}
\usepackage{enumitem}
\usepackage{tikz}

\newtheorem{theorem}{Theorem}[section]
\newtheorem{lemma}[theorem]{Lemma}
\newtheorem{proposition}[theorem]{Proposition}
\newtheorem{corollary}[theorem]{Corollary}

\theoremstyle{definition}
\newtheorem{definition}[theorem]{Definition}

\theoremstyle{remark}

\numberwithin{equation}{section}

\newcommand{\key}{\ensuremath{\kappa}}
\newcommand{\lock}{\ensuremath{\mathfrak{L}}}

\newcommand{\comp}[1]{\ensuremath{\mathbf{#1}}}
\newcommand{\wt}{\ensuremath{\mathrm{\mathbf{wt}}}}

\renewcommand{\aa}{\ensuremath{\comp{a}}}

\newcommand{\ee}{\ensuremath{\reflectbox{e}}}

\newcommand{\LKT}{\ensuremath{\mathrm{LKT}}}
\newcommand{\KKT}{\ensuremath{\mathrm{KKT}}}
\newcommand{\KD}{\ensuremath{\mathrm{KD}}}

\renewcommand{\flat}{\ensuremath{\mathrm{flat}}}

\newlength\cellsize 

\savebox2{%
\begin{picture}(10,10)
\put(0,0){\line(1,0){10}}
\put(0,0){\line(0,1){10}}
\put(10,0){\line(0,1){10}}
\put(0,10){\line(1,0){10}}
\end{picture}}

\newcommand\cellify[1]{\def\thearg{#1}\def\nothing{}%
\ifx\thearg\nothing\vrule width0pt height\cellsize depth0pt%
  \else\hbox to 0pt{\usebox2\hss}\fi%
  \vbox to 10\unitlength{\vss\hbox to 10\unitlength{\hss$#1$\hss}\vss}}

\newcommand\tableau[1]{\vtop{\let\\=\cr
\setlength\baselineskip{-10000pt}
\setlength\lineskiplimit{10000pt}
\setlength\lineskip{0pt}
\halign{&\cellify{##}\cr#1\crcr}}}

\usepackage{xcolor}
\usepackage[colorinlistoftodos]{todonotes}

\begin{document}


\title[Lock crystals]{Locks fit into keys: \\ a crystal analysis of Lock polynomials}  

\author[Wang]{George Wang}
\address{Department of Mathematics, University of Pennsylvania, 209 S. 33rd St., Philadelphia, PA 19104-6317, U.S.A.}
\email{wage@math.upenn.edu}
\thanks{Supported by NSF DGE-1845298.}




\begin{abstract}
Lock polynomials and lock Kohnert tableaux are natural analogues to key polynomials and key Kohnert tableaux, respectively.
In this paper, we compare lock polynomials to the much-studied key polynomials and show that the difference of a key polynomial and lock polynomial for the same composition is monomial positive. We also examine the conditions for which key and lock polynomials are symmetric or quasisymmetric. We accomplish these goals combinatorially using key Kohnert tableaux and lock Kohnert tableaux. In particular, for the difference of a key minus a lock, we focus on the behavior of crystal operators on Kohnert tableaux. The Type A Demazure crystal can be realized on the vertex set of key Kohnert tableaux, and we show with an explicit combinatorial definition that a similar crystal-like structure exists on the vertex set of lock Kohnert tableaux. Finally, we construct an injective, weight-preserving map from lock Kohnert tableaux to key Kohnert tableaux that intertwines the crystal operators.

\end{abstract}

\maketitle

%
\section{Introduction}
%
\label{sec:introduction}

Assaf and Searles \cite{AS18b} in their work on Kohnert diagrams and tableaux defined lock Kohnert tableaux and lock polynomials as analogues to key Kohnert tableaux and to the ubiquitous key polynomials. In this paper, we aim to ask and partially answer a natural question about such an analogue: what properties do locks and keys share, and how are they related to each other? We begin by examining the conditions for lock polynomials and key polynomials to be symmetric or quasisymmetric. These results follow from a careful examination of the combinatorial definitions of lock and key Kohnert tableaux. 

Our main result is that the difference of a key polynomial of a particular weak composition minus a lock polynomial of the same composition is monomial positive. We prove this result purely combinatorially using the Demazure crystal structure on key Kohnert tableaux, which we will call a key crystal, and an analagous, crystal-like structure on lock Kohnert tableaux that we will construct and refer to as a lock crystal. We prove that the lock crystal is connected and that there is an injective, weight-preserving algorithm from lock Kohnert tableaux to key Kohnert tableaux that intertwines their crystal operators. We accomplish this by utilizing the rectification operators of Assaf and Gonz\'alez in conjunction with our novel unlock operators. In particular, we will see that the unlock operators turn out to act on the underlying diagram of a lock Kohnert tableau in the same way as rectification operators, however, since unlock operators act on \emph{labeled} diagrams while rectification operators act on \emph{unlabeled} diagrams, unlock operators allow us to track the movements of labels through our algorithm. A nice consequence of this algorithm is that the lock crystal forms a subcrystal of a key crystal for the same weak composition, which in turn is known to be a subcrystal of a normal crystal.

%
\section{Key Polynomials}
%
\label{sec:keys}

There are many bases for the polynomial ring $\mathbb{Q}[x_1, \ldots, x_n]$ which have deep geometric and representation theoretic significance. We begin with one such basis by defining it combinatorially using diagrams indexed by weak compositions.

A \emph{diagram} is an array of finitely many cells in $\mathbb{N}\times\mathbb{N}$, and a \emph{weak composition} is an ordered sequence of nonnegative integers written $\comp{a} = (a_1,a_2,\ldots, a_n)$.
The weight of a diagram $D$, denoted $\wt(D)$, is the weak composition whose $i$th part is the number of cells in row $i$. A diagram is a \emph{key diagram} if the rows are left justified. For each weak composition $\comp{a}$, there is a unique key diagram of weight $\comp{a}$, which we simply call the \emph{key diagram of $\comp{a}$}.

Starting from a particular diagram $D$, one can generate new diagrams using \emph{Kohnert moves}. A Kohnert move on a diagram takes the rightmost cell of a given row and moves the cell to the first open position below, jumping over other cells if necessary. Let $\KD(D)$ denote the set of all diagrams that can be obtained from $D$ by a sequence of Kohnert moves.

In the case of key diagrams, we call the set of diagrams generated by Kohnert moves on the key diagram of $\aa$ the set of \emph{key Kohnert diagrams of $\comp{a}$}. Kohnert \cite{Koh91} showed that Demazure characters (or key polynomials) could be seen as the generating polynomials of key Kohnert diagrams of different weak compositions.

Assaf and Searles \cite{AS18b} gave a description of \emph{key Kohnert tableaux}, which they called Kohnert tableaux. These tableaux are unique labelings for key diagrams that track the original position of each cell in a Kohnert diagram before any Kohnert moves are applied.

\begin{definition}
Given a weak composition $\comp{a}$ of length $n$, a \emph{key Kohnert tableau of content $\comp{a}$} is a diagram filled with entries $1^{a_1}, 2^{a_2}, \ldots, n^{a_n}$, one per cell, satisfying the following conditions:

\begin{enumerate}
	\item there is exactly one $i$ in each column from $1$ through $a_i$;
	\item each entry in row $i$ is at least $i$;
	\item the cells with entry $i$ weakly descend from left to right;
	\item if $i<j$ appear in a column with $i$ above $j$, then there is an $i$ in the column immediately to the right of and strictly above $j$
\end{enumerate}
\end{definition}

The set of key Kohnert tableaux of content $\comp{a}$ is denoted $\KKT(\comp{a})$. We call condition (2) the \emph{flagged} condition and say that a labeled Kohnert diagram (not just key Kohnert tableaux) satisfying this condition is flagged.
An occurrence of (4) in any labeled Kohnert diagram is called an \emph{inversion} and we say that $i$ and $j$ are inverted. We also use the notation $\mathbb{D}(T)$ to denote the underlying diagram for a given labeled diagram $T$.

\begin{figure}[ht]
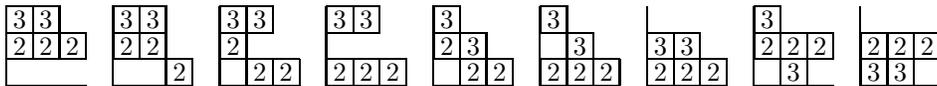

	\begin{displaymath}
	\begin{array}{ccccc cccc}
	\vline\tableau{3 & 3 \\ 2 & 2 & 2 \\ & \\\hline}&
	\vline\tableau{3 & 3 \\ 2 & 2 \\ & & 2 \\ \hline}&
	\vline\tableau{3 & 3 \\ 2 \\ & 2 & 2 \\ \hline}&
	\vline\tableau{3 & 3 \\ & \\ 2 & 2 & 2 \\ \hline}&
	\vline\tableau{3 \\ 2 & 3 \\ & 2 & 2\\ \hline}&
	\vline\tableau{3 \\ & 3 \\ 2 & 2 & 2 \\ \hline}&
	\vline\tableau{& \\ 3 & 3 \\ 2 & 2 & 2 \\ \hline}&
	\vline\tableau{3 & \\ 2 & 2 & 2 \\ & 3 \\ \hline}&
	\vline\tableau{& \\ 2 & 2 & 2 \\ 3 & 3 \\ \hline}
	\end{array}
	\end{displaymath}
	\caption{The set $\KKT(0,3,2)$.}
	\label{fig:KKT}
\end{figure}

Since each key Kohnert diagram has a unique such labeling, we may define key polynomials as generating polynomials over key Kohnert tableaux instead.

\begin{definition}
The \emph{key polynomial} indexed by the weak composition $\comp{a}$ is
\begin{equation}
\key_\comp{a} = \sum_{T\in \KKT(\comp{a})} x^{\wt(T)}
\end{equation}

\end{definition}

For example, we have from Figure \ref{fig:KKT} that
\begin{displaymath}
  \key_{(1,0,2,1)} = x_1^2x_2x_3 + x_1x_2^2x_3 + x_1x_2x_3^2 + x_1^2x_2x_4 + x_1x_2^2x_4 + x_1^2x_3x_4 + x_1x_2x_3x_4 + x_1x_3^2x_4 .
\end{displaymath}

Key polynomials are a polynomial generalization of the \emph{Schur polynomials}, which are an important basis of the \emph{symmetric polynomials}. Symmetric polynomials are those that are invariant under permutations of variable indices, and the theory of symmetric polynomials is a rich and beautiful subject, expertly introduced in \cite{Mac95, EC2}. We also have \emph{quasisymmetric polynomials} that lie between symmetric polynomials and the full polynomial ring. 
A polynomial is \emph{quasisymmetric} if the coefficients of any two monomials agree whenever their ordered sequence of nonzero exponents agree.

Macdonald \cite{Mac91} first observed that if $\comp{a}$ is weakly increasing, then the corresponding key polynomial is a Schur polynomial and therefore symmetric, which also follows from the more general result of Assaf and Searles \cite{AS18b}(Theorem~4.2).

\begin{theorem}[\cite{AS18b}]
  For a weak composition $\comp{a}$ of length $n$, the key polynomial $\key_{\comp{a}}$ is symmetric in $x_1,\ldots,x_n$ if and only if $\comp{a}$ is weakly increasing. Moreover, in this case, $\key_\aa = s_{\mathrm{rev}(\aa)}(x_1,\ldots,x_n)$.
  \label{thm:key-sym}
\end{theorem}

We can also characterize directly when a key polynomial is quasisymmetric. 

\begin{proposition}\label{prop:key-qsym}
  For a weak composition $\comp{a}$ of length $n$, the key polynomial $\key_\comp{a}$ is quasisymmetric in $x_1, x_2,\ldots,x_n$ if and only if $\comp{a}$ has no zero parts or the parts are weakly increasing. 
\end{proposition}

\begin{proof}
We first consider when $\aa$ is weakly increasing. By Theorem \ref{thm:key-sym}, $\key_\aa$ is symmetric and so it is also quasisymmetric.

Next suppose that $\aa$ has no zero parts. The diagram of $\aa$ has a box in every row from $1$ to $n$ in the leftmost column, and any sequence of Kohnert moves preserves this property. Then $x^{\wt(T)}$ for any key Kohnert tableau $T$ of $\aa$ has positive exponent for $x_1,\ldots, x_n$ and is as a result quasisymmetric in $x_1,\ldots,x_n$. Therefore, $\key_\aa$ is a sum of quasisymmetric monomials and is also quasisymmetric.

Finally, suppose that $\aa$ is not weakly increasing and has at least one part equal to zero. We consider two cases: either there exists some index $i$ for which $a_i>a_{i+1} = 0$, or there does not. 

Suppose first that such an index exists. Observe that for a given diagram $D$, $\wt(D)$ comes in later lexicographic order than the weights of any diagrams resulting from a sequence of Kohnert moves on $D$. Then since $\key_\aa$ contains the term
$$x_1^{a_1} \cdots x_i^{a_i} x_{i+1}^{a_{i+1}}\cdots x_n^{a_n} = x_1^{a_1} \cdots x_i^{a_i} x_{i+1}^{0}\cdots x_n^{a_n}$$
but not the term
$x_1^{a_1} \cdots x_i^{0} x_{i+1}^{a_{i}}\cdots x_n^{a_n},$
$\key_\aa$ is not quasisymmetric.

Now suppose that no such index $i$ exists, so that $\aa$ has some positive number of leading zeroes followed by exclusively nonzero parts. Choose $j$ such that $a_j > a_{j+1} > 0$. We can apply Kohnert moves to the diagram of $\aa$ to push all nonempty rows below row $j$ down by exactly one space, then apply $a_{j+1}$ Kohnert moves to row $j+1$ to move the boxes in row $j+1$ to row $j-1$. Now we have a key Kohnert diagram with associated monomial 
$$x_1^{a_2} \cdots x_{j-2}^{a_{j-1}} x_{j_1}^{a_{j+1}} x_j^{a_j} x_{j+1}^0 x_{j+2}^{a_{j+2}} \cdots x_n^{a_n}.$$
If $\key_aa$ were quasisymmetric, then we would also need the monomial 
$$x_1^{a_1} \cdots x_{j-1}^{a_{j-1}} x_j^{a_{j+1}} x_{j+1}^{a_j} x_{j+2}^{a_{j+2}} \cdots x_n^{a_n}.$$
However, the weight of the key Kohnert diagram that this monomial would be associated with would have a later lexicographic order than $\aa$, which contradicts our observation above that Kohnert moves on a diagram must produce weights with an earlier lexicographic order. Therefore, $\key_\aa$ is not quasisymmetric.

\end{proof}

Notably, the only key polynomials that are quasisymmetric but not symmetric are those with nonzero parts not weakly increasing and with no zero parts.

\section{Lock polynomials}
\label{sec:locks}

Assaf and Searles \cite{AS17} introduced \emph{lock polynomials} as a natural analogue to the combinatorial definition of key polynomials. Given a weak composition $\aa$, the \emph{lock diagram} of $\aa$ is the unique right justified diagram with weight $\aa$. The \emph{lock Kohnert diagrams} of $\aa$ are all diagrams that can be obtained from applying a sequence of Kohnert moves to the lock diagram of $\aa$. Lock Kohnert diagrams similarly have unique labelings, which we call \emph{lock Kohnert tableaux}.

\begin{definition}[\cite{AS17}]
Given a weak composition of length $n$, a \emph{lock Kohnert tableau of content $\aa$} is a diagram filled with entries $1^{a_1}, 2^{a_2}, \ldots, n^{a_n}$, one per cell, satisfying the following conditions:
\begin{enumerate}
	\item there is exactly one $i$ in each column from $\max(\aa) - a_i+1$ through $\max(\aa)$;
	\item each entry in row $i$ is at least $i$;
	\item the cells with entry $i$ weakly descend from left to right;
	\item the labeling strictly decreases down columns.
\end{enumerate}
\end{definition}
These are unique labelings because condition (1) fixes the set of labels in each column and condition (4) fixes their order within each column. The set of lock Kohnert tableaux of content $\aa$ is denoted $\LKT(\aa)$, and we reiterate from earlier that we use the notation $\mathbb{D}(T)$ to denote the underlying diagram for a given labeled diagram $T$. We also use $T_\aa$ to denote the unique lock Kohnert tableau with weight $\flat(\aa)$ and content $\aa$.

\begin{figure}[ht]
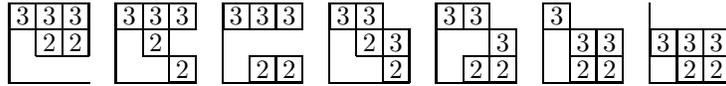


\begin{displaymath}
\begin{array}{ccccc cc}
	\vline\tableau{3 & 3 & 3 \\ & 2 & 2 \\ & \\ \hline}&
	\vline\tableau{3 & 3 & 3 \\ & 2 \\ & & 2\\ \hline}&
	\vline\tableau{3 & 3 & 3 \\ & \\ & 2 & 2\\ \hline}&
	\vline\tableau{3 & 3\\ & 2 & 3\\ & & 2\\ \hline}&
	\vline\tableau{3 & 3 \\ & & 3 \\ & 2 & 2 \\ \hline}&
	\vline\tableau{3\\ & 3 & 3\\ & 2 & 2\\ \hline}&
	\vline\tableau{& \\ 3 & 3 & 3 \\ & 2 & 2 \\ \hline}
\end{array}
\end{displaymath}
\caption{The set $\LKT(0,2,3)$.}
\label{fig:LKT}
\end{figure}

We define lock polynomials as the generating polynomials of lock Kohnert tableaux.

\begin{definition}[\cite{AS17} \label{def:LKT}]
The \emph{lock polynomial} indexed by the weak composition $\aa$ is
\begin{equation}
\lock_\aa = \sum_{T\in \LKT(\aa)} x^{\wt(T)}.
\end{equation}
\end{definition}
For example, we have from Figure \ref{fig:LKT} that
$$\lock_{(0,2,3)} = x_2^2 x_3^3 + x_1 x_2 x_3^3 + x_1^2 x_3^3 + x_1 x_2^2 x_3^2 + x_1^2 x_2 x_3^2 + x_1^2 x_2^2 x_3 + x_1^2 x_2^3.$$
Lock polynomials also form a basis for the full polynomial ring \cite{AS17}, and they coincide with key polynomials if the nonzero parts of $\aa$ are weakly decreasing.

\begin{theorem}[\cite{AS17}, Theorem 6.12]
Given a weak composition $\aa$ such that its nonzero parts are weakly decreasing, we have
\begin{equation}
\lock_\aa = \key_\aa.
\end{equation}
\label{thm:lock=key}
\end{theorem}

As with key polynomials, lock polynomials are not always symmetric or quasisymmetric, however we can characterize exactly when each happens. For the quasisymmetric case, the condition is the same as for key polynomials.

\begin{proposition} \label{prop:lock-qsym}
For $\aa$ a weak composition of length $n$, $\lock_\aa$ is quasisymmetric in $x_1, x_2,\ldots, x_n$ if and only if $\aa$ has no zero parts or the parts are weakly increasing.
\end{proposition}

\begin{proof}
If there are no zero parts, then no Kohnert moves can be done on the lock diagram of $\aa$. Therefore, $\lock_\aa$ consists of a single monomial with positive exponent for all variables $x_1,\ldots,x_n$, and therefore $\lock_\aa$ is quasisymmetric.

Now suppose that $\aa$ is weaky increasing with leading zeroes. Define maps $p_i$ and $d_i$ for $1 \leq i < n$ as follows. If row $i$ (row $i+1$) has at least one box in it and row $i+1$ (row $i$) is empty, $p_i$ ($d_i$) moves all boxes from row $i$ (row $i+1$) to row $i+1$ (row $i$), preserving their columns and labels, otherwise $p_i$ ($d_i$) does nothing. 
We can think of these as colored edges connecting different labeled diagrams, where a connected component has generating polynomial equal to a monomial quasisymmetric polynomial in $n$ variables. Therefore, it is sufficient to show that when at least one labeled diagram in a connected component is a lock Kohnert tableau with content $\aa$, every labeled diagram in that connected component is a lock Kohnert tableau with content $\aa$, since then summing over the connected components with lock Kohnert tableau gives the lock polynomial as a sum of monomial quasisymmetrics.

When $d_i$ is applied to a lock Kohnert tableau of content $\aa$, it is easy to check that all four properties in Definition \ref{def:LKT} are preserved. For $p_i$, properties (1), (3), and (4) are also clear by construction. For property (2), suppose that some box in column $j$ with label $i$ is pushed to row $i+1$ by $p_i$. By properties (1) and (2) and the fact that $\aa$ is weakly increasing, there must be boxes with labels $i+1, i+2, \ldots, n$ strictly above row $i+1$ in column $j$. However, since there cannot be boxes above row $n$, we must have $n-i$ boxes fitting into $n-i-1$ rows, which is impossible. Therefore, property (2) must also hold, and any labeled diagram connected to a lock Kohnert tableau of content $\aa$ by a sequence of $p_i, d_i$ is also a lock Kohnert diagram of content $\aa$. 

Finally, consider the case where the parts of $\aa$ are not weakly increasing and at least one part is equal to zero. The proof in this case is essentially identical to that of the same case in the proof of Proposition \ref{prop:key-qsym} and the analagous conclusion follows, that the lock polynomial of $\aa$ is not quasisymmetric in this case.

\end{proof}

Symmetry for lock polynomials is less common than for key polynomials, as seen by comparing Theorem \ref{thm:key-sym} with the following.

\begin{proposition}
For $\aa$ a weak composition of length $n$, the lock polynomial $\lock_\aa$ is symmetric in $x_1,x_2,\ldots,x_n$ if and only if $\aa = 0^{n-k}\times m^k$ for some integers $m,k >0$ and $k \leq n$. Moreover, in this case, we have $\lock_\aa = s_{m^k}(x_1,\ldots,x_n)$.
\end{proposition}

\begin{proof}
By Proposition \ref{prop:lock-qsym}, $\aa$ must be weakly increasing or else $\lock_\aa$ is not quasisymmetric, and so not symmetric.

Suppose then that $\aa$ is weakly increasing and that there exists some index $i$ such that $a_{i+1} > a_i > 0$, and let $s_i\aa$ be $\aa$ with the parts $a_i$ and $a_{i+1}$ swapped. The lock polynomial of $\aa$ must contain a monomial $x^\aa$, so if it is symmetric, it must also contain the monomial $x^{s_i\aa}$. Consider a lock Kohnert tableau that would be associated with this monomial. 

By condition (2) in Definition \ref{def:LKT}, every box in rows $i+2$ to $n$ must have a label between $i+2$ and $n$, and since there are $a_{i+2} + \cdots + a_{n}$ many such boxes and labels, every such box must have such a label, and there are no remaining labels between $i+2$ and $n$ to place in lower rows.

Using condition (2) again, every one of the $a_i$ boxes in row $i+1$ must have an $i+1$ label, since no smaller labels can exist in row $i+1$, and from above, no larger labels can either. Since $a_{i+1} > a_i$, this leaves $a_{i+1}-a_i$ many $i+1$ labels that must go in lower rows. Since columns strictly decrease, these excess $i+1$ labels must be to the left of column $\max(\aa) - a_i +1$. However, this would imply the existence of $i+1$ labels strictly lower and to the left of the $i+1$ labels in row $i+1$, which contradicts condition (3). Therefore, no such lock Kohnert tableau can exist, and $\lock_\aa$ is not symmetric.

The only remaining cases are those for which $\aa = 0^{n-k} \times m^k$. By Theorem \ref{thm:lock=key}, we have $\lock_\aa = \key_\aa$, then by Theorem \ref{thm:key-sym}, we have $\key_\aa = s_{\mathrm{rev}(\aa)}(x_1,\ldots,x_n)$, so $\lock_\aa$ is always symmetric in these cases.
\end{proof}

%
\section{Crystals on Kohnert tableaux}
%
\label{sec:crystals}

Kashiwara \cite{Kas91} introduced the notion of crystal bases in his study of the representation theory of quantized universal enveloping algebras at $q=0$. He proved that normal crystals correspond to polynomial representations of the general linear group, with connected crystals corresponding to the irreducible representations. The characters of these connected normal crystals are the celebrated Schur functions. Kashiwara and Nakashima \cite{KN94} and Littelmann \cite{Lit95} gave an explicit combinatorial construction of normal crystals with explicit raising and lowering operators that act on semistandard Young tableaux.

Demazure \cite{Dem74b} introduced Demazure characters (which we refer to as key polynomials) that arose in connection with Schubert calculus \cite{Dem74a}. Demazure crystals (which we refer to as key crystals) are certain truncations of normal crystals conjectured by Littelmann \cite{Lit95} and proved by Kashiwara \cite{Kas93} to categorify Demazure characters. A subset $X$ of a crystal $\mathcal{B}$ has an induced structure taken from the crystal on $\mathcal{B}$ that we refer to as a \emph{subcrystal} of $\mathcal{B}$, and so a key crystal is a subcrystal of the normal crystal that it is a truncation of.

Combinatorially for the general linear group (type A), a \emph{crystal basis} is a set $\mathcal{B}$ not containing $0$, a \emph{weight map} $\wt:\mathcal{B} \rightarrow \mathbb{Z}^n$ from the basis to the weight lattice, and \emph{raising and lowering operators} $e_i,f_i: \mathcal{B} \rightarrow \mathcal{B} \cup \{0\}$, for $i = 1, 2, \ldots, n-1$ that satisfy certain axioms including $e_i(b) = b'$ if and only if $f_i(b') = b$. In particular, to define a crystal, it is enough to define the set $\mathcal{B}$, the weight map $\wt$, and the raising operators $e_i$, from which the lowering operators can be deduced.

Assaf and Schilling \cite{AS18a}(Definition 3.7) gave an explicit combinatorial construction of key crystals with raising and lowering operators that act on semistandard key tableaux, which can be translated into the language of Kohnert diagrams and tableaux, as presented in \cite{AG19} by Assaf and Gonz\'alez. 
In this paper, we focus specifically on these crystal operators on Kohnert diagrams and tableaux.

\begin{definition}[\cite{AG19}]\label{def:kohnert-pairing}
Given any diagram $D$ with $n\geq 1$ rows and $1 \leq i < n$, define the \emph{vertical $i$-pairing} of $D$ as follows: $i$-pair any boxes in rows $i$ and $i+1$ that are located in the same column and then interatively vertically $i$-pair any unpaired boxes in rows $i+1$ with the rightmost unpaired box in row $i$ located in a column to its left whenever all the boxes in rows $i$ and $i+1$ in the columns between them are already vertically $i$-paired.
\end{definition}

\begin{definition}[\cite{AG19}]
Given any integer $n \geq 0$ and any diagram $D$ with at most $n$ rows, for any integer $1 \leq i < n$, define the \emph{raising operator} $e_i$ on the space of diagrams as the operator that pushes the rightmost vertically unpaired box in row $i+1$ of $D$ down to row $i$. If $D$ has no vertically unpaired boxes in row $i+1$, then $e_i(D) = 0$. 
\end{definition}

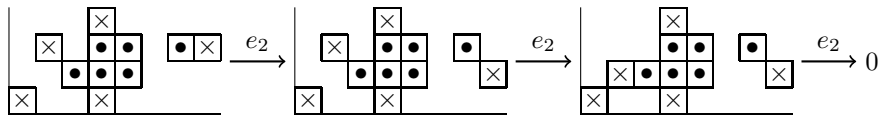
\begin{figure}[ht]
    \begin{center}
        \begin{tikzpicture}[xscale=3.8,yscale=1.6]
	\draw (-0.37,0.56)--(-0.37,1.45);
	\draw (0.63,0.56)--(0.63,1.45);
	\draw (1.63,0.56)--(1.63,1.45);
		\node at (0,1) (T01) {$\tableau{ 
        & & & \times \\
        & \times & & \bullet & \bullet & & \bullet & \times \\
        & & \bullet & \bullet & \bullet \\
        \times & & & \times\\ \hline
        }$};
		\node at (1,1) (T11) {$\tableau{ 
        & & & \times \\
        & \times & & \bullet & \bullet & & \bullet &  \\
        & & \bullet & \bullet & \bullet & & & \times \\
        \times & & & \times\\ \hline 
        }$};
		\node at (2,1) (T21) {$\tableau{ 
        & & & \times \\
        &  & & \bullet & \bullet & & \bullet &  \\
        & \times & \bullet & \bullet & \bullet & & & \times\\
        \times & & & \times\\ \hline
        }$};
        \node at (2.65,1) (Z) {$0$};
		\draw[thick,->] (T01) -- (T11) node[midway,above] {$e_2$};
		\draw[thick,->] (T11) -- (T21) node[midway,above] {$e_2$};
		\draw[thick,->] (T21) -- (Z) node[midway,above] {$e_2$};

	\end{tikzpicture}
    \end{center}
    \caption{The vertically $2$-paired boxes are highlighted with dots, while the remaining boxes in the third row of the leftmost diagram are unpaired.}
    \label{fig:vert-pair}
\end{figure}

Assaf and Gonz\'alez show in \cite{AG19}(Proposition 5.23) that these raising operators on key Kohnert diagrams coincide with their raising operators on key Kohnert tableaux. Therefore, we can simply take the definition of raising operators on key Kohnert tableaux as follows.

\begin{definition}
Given $T$ a key Kohnert tableau of content $\aa$ with underlying diagram $D$, if $e_i(D) \neq 0$, then the \emph{raising operator} $e_i$ on $T$ produces the unique key Kohnert tableau of content $\aa$ with underlying diagram $e_i(D)$. Otherwise, $e_i(T) = 0$.
\end{definition}

We note that there is a labeling algorithm \cite{AS18b} that constructs the unique key Kohnert tableau for a given key Kohnert diagram, as well as a more direct construction of crystal operators for keys \cite{AS18a}, but the specifics of these are not necessary for this paper.

We can repeat a similar process for raising operators on lock Kohnert tableaux. That is, given $T$ a lock Kohnert tableau of content $\aa$ with underlying diagram $D$, the raising operator on $T$ produces the unique lock Kohnert tableau of content $\aa$ with underlying diagram $e_i(D)$ if it exists, otherwise $e_i(T)=0$. Note that in this case, we also specify that the resulting diagram must have a valid lock Kohnert tableau labeling. This is because while the raising operator $e_i$ on a key Kohnert tableau always produces another key Kohnert tableau of the same content, the same may not be true for a given lock Kohnert tableau. Put another way, the minimal $k$ such that $e_i^{k+1}(T)= 0$ is the number of unpaired boxes in row $i+1$ for a key Kohnert tableau but may be smaller for a lock Kohnert tableau.

We also provide the following equivalent formulation for raising operators on lock Kohnert tableaux for completeness, where boxes are vertically paired based on the underlying diagram.

\begin{definition}
  Given a weak composition $\aa$, $T \in \LKT(\aa)$, and $1 \leq i < n$, the \emph{raising operator} $e_i$ acts on $T$ by $e_i(T) = 0$ if $T$ has no vertically unpaired boxes in row $i+1$ or if the rightmost unpaired box in row $i+1$ has the same label as a box to its right in the same row. Otherwise, $e_i$ pushes the rightmost vertically unpaired box in row $i+1$ of $T$ down to row $i$.
\end{definition}

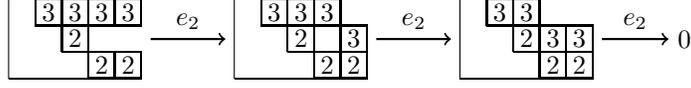
\begin{figure}
    \begin{center}
        \begin{tikzpicture}[xscale=3,yscale=1.6]
	\draw (-.295,0.67)--(-.295,1.335);
	\draw (1-.295,0.67)--(1-.295,1.335);
	\draw (2-.295,0.67)--(2-.295,1.335);
		\node at (0,1) (T01) {$\tableau{
		 & 3 & 3 & 3 & 3 \\
		 & & 2 \\
		 & & & 2 & 2 \\ \hline
		}$};
		\node at (1,1) (T11) {$\tableau{
		 & 3 & 3 & 3  \\
		 & & 2 & & 3\\
		 & & & 2 & 2 \\ \hline
		}$};
		\node at (2,1) (T21) {$\tableau{
		 & 3 & 3 &  \\
		 & & 2 & 3 & 3\\
		 & & & 2 & 2 \\ \hline
		}$};
        \node at (2.7,1) (Z) {$0$};
		\draw[thick,->] (T01) -- (T11) node[midway,above] {$e_2$};
		\draw[thick,->] (T11) -- (T21) node[midway,above] {$e_2$};
		\draw[thick,->] (T21) -- (Z) node[midway,above] {$e_2$};

	\end{tikzpicture}
    \end{center}
    \caption{Raising operators acting on a lock Kohnert tableau of content $(0,3,4)$. Notice that the third tableau is sent to zero despite there being an unpaired box in row 3.}
    \label{fig:LKT-raising}
\end{figure}

It is straightforward to see that this coincides with the previous definition on the underlying diagram. 
To avoid excessive notation, we will use $e_i$ for any raising operator on lock or key Kohnert diagrams or tableaux and $f_i$ for any lowering operator, where the type of object being acted on will either be clear from context or specified if not.

A \emph{crystal graph} is a visual depiction of a crystal, using the crystal basis as the vertices. These vertices are then connected by colored edges representing crystal operators. We will refer to the crystal graph on key Kohnert tableaux of content $\aa$ as the \emph{key crystal of $\aa$} and the crystal graph on lock Kohnert tableaux of content $\aa$ as the \emph{lock crystal of $\aa$} (see Figure \ref{fig:crystals} for example). It is well-known that the key crystal is connected, and it turns out that the same is true for the lock crystal.

\begin{figure}[ht]
    \begin{center}
        \begin{tikzpicture}[xscale = 2, yscale = 2]

        \node at (2,5) (A) {$\tableau{ & \\ 4 \\ 3 \\ 1 & 3 \\ \hline}$};
	\draw (2-0.178,4+0.65)--(2-0.178,4+1.35);
        \node at (1,4) (B) {$\tableau{ & \\ 4 \\ 3 & 3 \\ 1 & \\ \hline}$};
	\draw (1-0.178,3+0.65)--(1-0.178,3+1.35);
        \node at (3,4) (C) {$\tableau{ 4 & \\  \\ 3 \\ 1 & 3 \\ \hline}$};
	\draw (3-0.178,3+0.65)--(3-0.178,3+1.35);
        \node at (1,3) (D) {$\tableau{  & \\ 3 & 3 \\ 4 \\ 1  \\ \hline}$};
	\draw (1-0.178,2+0.65)--(1-0.178,2+1.35);
        \node at (2,3) (E) {$\tableau{ 4 & \\ & \\ 3 & 3 \\ 1  \\ \hline}$};
	\draw (2-0.178,2+0.65)--(2-0.178,2+1.35);
        \node at (3,3) (F) {$\tableau{ 4 & \\ 3 \\ \\ 1 & 3 \\ \hline}$};
	\draw (3-0.178,2+0.65)--(3-0.178,2+1.35);
        \node at (2,2) (G) {$\tableau{ 4 & \\ 3 & \\  & 3 \\ 1  \\ \hline}$};
	\draw (2-0.178,1+0.65)--(2-0.178,1+1.35);
        \node at (2,1) (H) {$\tableau{ 4 & \\ 3 & 3 \\ & \\ 1  \\ \hline}$};
	\draw (2-0.178,0.65)--(2-0.178,1.35);
        \node at (4.5, 4) (I) {$\tableau{ & \\ & 4 \\ 3 & 3 \\ & 1 \\ \hline}$};
	\draw (4.5-0.178,3+0.65)--(4.5-0.178,3+1.35);
        \node at (4.5, 3) (J) {$\tableau{ & \\ 3 & 4 \\ & 3 \\ & 1 \\ \hline}$};
	\draw (4.5-0.178,2+0.65)--(4.5-0.178,2+1.35);
        \node at (5.5, 3) (K) {$\tableau{ & 4 \\ & \\ 3 & 3 \\ & 1 \\ \hline}$};
	\draw (5.5-0.178,2+0.65)--(5.5-0.178,2+1.35);
        \node at (5.5, 2) (L) {$\tableau{ & 4 \\ 3& \\  & 3 \\ & 1 \\ \hline}$};
	\draw (5.5-0.178,1+0.65)--(5.5-0.178,1+1.35);
        \node at (5.5, 1) (M) {$\tableau{ & 4 \\ 3& 3 \\  &  \\ & 1 \\ \hline}$};
	\draw (5.5-0.178,0.65)--(5.5-0.178,1.35);
        \draw[thick,->,blue] (A) -- (B) node[midway,above] {$1$};
        \draw[thick,->,blue] (C) -- (E) node[midway,above] {$1$};
        \draw[thick,->,purple] (B) -- (D) node[midway,left] {$2$};
        \draw[thick,->,purple] (C) -- (F) node[midway,right] {$2$};
        \draw[thick,->,purple] (E) -- (G) node[midway,left] {$2$};
        \draw[thick,->,purple] (G) -- (H) node[midway,left] {$2$};
        \draw[thick,->,violet] (A) -- (C) node[midway,above] {$3$};
        \draw[thick,->,violet] (B) -- (E) node[midway,above] {$3$};
        \draw[thick,->,purple] (I) -- (J) node[midway,left] {$2$};
        \draw[thick,->,purple] (K) -- (L) node[midway,left] {$2$};
        \draw[thick,->,purple] (L) -- (M) node[midway,left] {$2$};
        \draw[thick,->,violet] (I) -- (K) node[midway,above] {$3$};
        \end{tikzpicture}
    \end{center}
    \caption{On the left is the key crystal of $\aa = (1, 0, 2, 1)$ and on the right is the lock crystal of $\aa$.}
    \label{fig:crystals}
\end{figure}
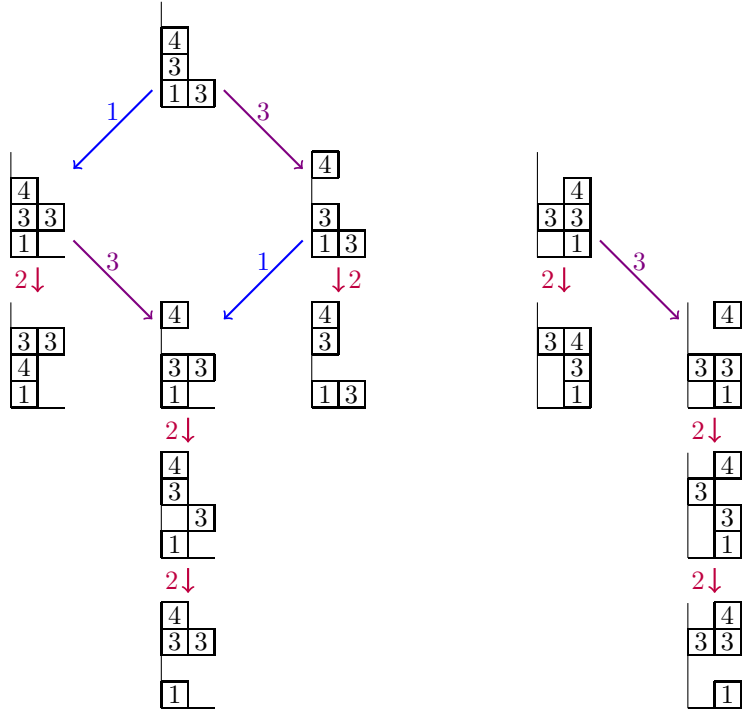

\begin{theorem}\label{thm:lock-connected}
For $\aa$ a weak composition, the raising and lowering operators on semistandard lock tableaux generate a connected, colored graph on $\LKT(\aa)$. 
\end{theorem}

\begin{proof}
See Figure \ref{fig:proof-example} for an explicit example of the argument below.
Recall that $T_\aa$ denotes the LKT of content $\aa$ with weight $\flat(\aa)$. We can check that this is unique by the definition of lock Kohnert tableaux. It is sufficient to show that for $T\in \LKT(\aa)$ with highest box in row $m$, $T$ is connected to $T_\aa$ using only the crystal operators $e_1, f_1, \ldots, e_{m-1}, f_{m-1}$. We prove this by inducting on the size of $\aa$. The base case consists of weak compositions of size $1$, where the single box in row $m$ is always connected to the single box in row $1$ by the sequence of crystal operators $e_{1}\circ  e_2 \circ \cdots \circ e_{m-1}$. 

Assume now that for any weak composition $\comp{b}$ of size $n-1$ or smaller, if $S\in \LKT(\comp{b})$ with highest box in row $i$, we can connect $S$ to $T_\comp{b}$ using only the operators $e_1,f_1,\ldots,e_{i-1},f_{i-1}$. Fix $\comp{a}$ to be a weak composition of size $n$ with nonzero parts $\{a_{j_1},\ldots, a_{j_k}\}$, and let $T\in \LKT(\aa)$ with highest box in row $m$.

Let $T'$ be the $\LKT$ obtained by removing all boxes of $T$ in row $m$, and let $T'$ have shape $\aa '$ and highest box in row $m' < m$. By the inductive assumption, there is some sequence of crystal operators $e_1,f_1,\ldots, e_{m'-1}, f_{m'-1}$ that sends $T'$ to $T_{\aa '}$. Since these crystal operators only check the positions of boxes in row $m'$ and below, applying the same sequence of crystal operators to $T$ gives a tableau $U$ which boxes in row $m$ everywhere that $T$ does and which has boxes below row $m$ everywhere that $T_{\aa '}$ does. Suppose that $U$ has some number $t$ of boxes with label $j_k$ in row $k$. There are no boxes in the rows strictly between $k$ and $m$ and every box with label $j_k$ in row $k$ must be strictly right of every box in row $m$, so applying $f^t_{m-1} \circ\ \cdots \circ f^t_k$ to $U$ brings all $t$ of the boxes that were in row $k$ with label $j_k$ to row $m$. Therefore, we can assume every box of $U$ with label $j_k$ must be in row $m$.

If the only boxes in row $m$ have label $j_k$, then set $W=U$. Otherwise, $U$ has some boxes in row $m$ with label smaller than $j_k$, so let $c$ be the rightmost column containing such a box and let that box have label $\ell$. Obtain $U'$ from $U$ by removing all boxes in row $m$, then obtain $V'$ from $U'$ by pushing the highest box of each column to the right of $c$ up to row $m-1$ while preserving their label. This clearly still satisfies the column strict condition on LKT, and the sets of labels in each column are unchanged so condition (1) holds as well.

Suppose that condition (3) of Definition \ref{def:LKT} is not satisfied in $V'$ because of some pair of boxes $x$ left of $y$ with label $p \neq \ell$, where $y$ is pushed above $x$. By construction, $x$ must be weakly left of column $c$ and $y$ must be strictly right. 
In $U$, column $c$ contains a box in row $m$ with label $\ell$ and no boxes above row $m$. Then the column strict condition implies that there cannot be any labels larger than $\ell$ in column $c$, and then condition (1) implies there cannot be labels to the right of $c$ with label larger than $\ell$ either. It also implies that the highest box in each column to the right of $c$ in $U'$ must have label at least $\ell$. Therefore, if $p >\ell$, then $x$ cannot exist, and if $p < \ell$, then $y$ is not the highest box in its column is is therefore not pushed upwards.

Since $U$ is an $\LKT$ with a label $\ell$ in row $m$, we have $\ell \geq m$. Then using the observation that the highest box in each column to the right of $c$ in $U'$ has label at least $\ell$, every box that is pushed up to row $m-1$ in $V'$ has label at least $m-1$. Since all conditions are satisfied, $V'$ is an $\LKT$ by definition. Then by the inductive assumption, some sequence of the operators $e_1, f_1,\ldots, e_{m-2}, f_{m-2}$ sends $U'$ to $V'$, and therefore the same sequence of operators on $U$ gives a tableau $V$ which has boxes in row $m$ everywhere that $U$ does and which has boxes below row $m$ everywhere that $V'$ does.

By construction, all boxes in row $m$ with label $j_k$ of $V$ must be paired and all other boxes of row $m$, which have label smaller than $j_k$, are unpaired. We can then apply $e_{m-1}$ operators until all the unpaired boxes of row $m$ are in row $m-1$ and call the new tableau $W$.

In either case, the tableau $W$ has every box with label $j_k$ in row $m$ and every box with label smaller than $j_k$ below row $m$. Obtain $W'$ a LKT of content $\aa''$ from $W$ by removing all boxes in row $m$. By the inductive assumption, some sequence of crystal operators $e_1,f_1,\ldots, e_{m-2},f_{m-2}$ sends $W'$ to $T_{\comp{a}''}$. Then applying the same sequence of operators to $W$ followed by applying $e^{a_{j_k}}_k \circ\cdots\circ e^{a_{j_k}}_{m-1}$ sends $W$ to $T_\comp{a}$ and we are done.

\end{proof}

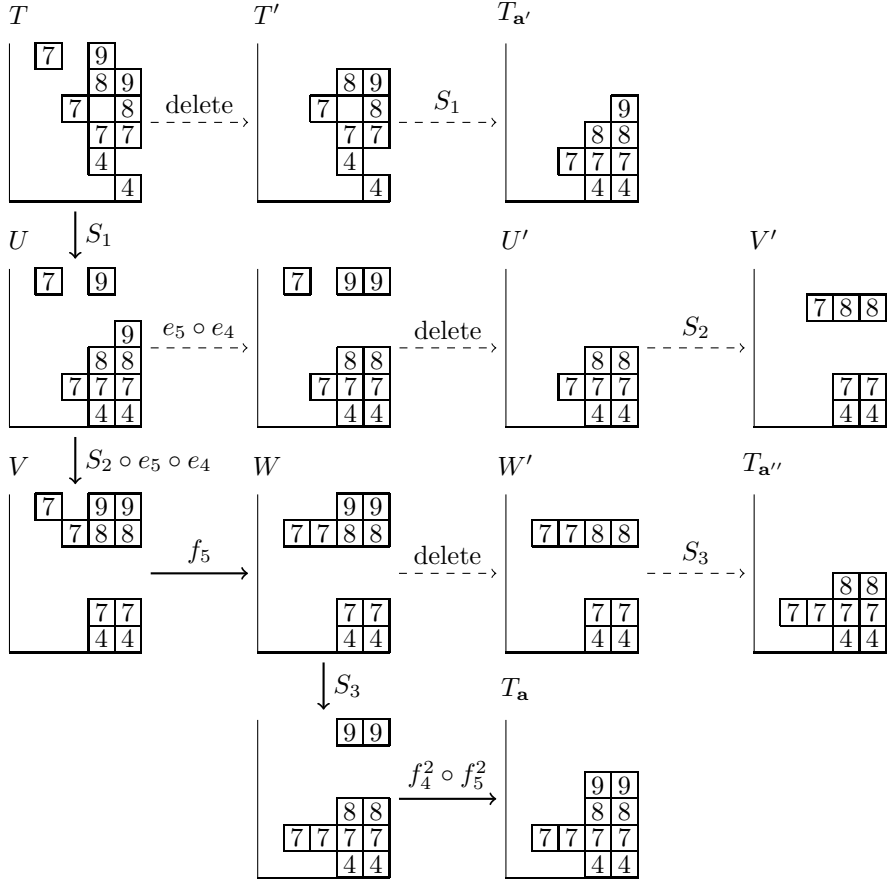
\begin{figure}[ht]
    \begin{center}
        \begin{tikzpicture}[xscale = 3.3, yscale = 3]
            \node at (1,4) (A) [label = {[xshift = -0.75cm] $T$}]{$\tableau{& 7 & & 9 & \\ & & & 8 & 9 \\ & & 7 & & 8 \\ & & & 7 & 7 \\ & & & 4 & \\ & & & & 4 \\ \hline}$};
	\draw (1-0.265,3+0.65)--(1-0.265,3+1.35);
            \node at (2,4) (B) [label = {[xshift = -0.75cm] $T'$}]{$\tableau{&  & &  & \\ & & & 8 & 9 \\ & & 7 & & 8 \\ & & & 7 & 7 \\ & & & 4 & \\ & & & & 4 \\ \hline}$};
	\draw (2-0.265,3+0.65)--(2-0.265,3+1.35);
            \node at (3,4) (C) [label = {[xshift = -0.75cm] $T_{\aa'}$}]{$\tableau{& \\ & \\ & & & & 9 \\ & & & 8 & 8\\ & & 7 & 7 & 7 \\ & & & 4 & 4\\ \hline}$};
	\draw (3-0.265,3+0.65)--(3-0.265,3+1.35);
            \node at (1,3) (D) [label = {[xshift = -0.75cm] $U$}] {$\tableau{& 7 & & 9 & \\ & \\ & & & & 9 \\ & & & 8 & 8\\ & & 7 & 7 & 7 \\ & & & 4 & 4\\ \hline}$};
	\draw (1-0.265,2+0.65)--(1-0.265,2+1.35);
            \node at (2,3) (E) {$\vline\tableau{& 7 & & 9 & 9 \\ & \\ & & & & \\ & & & 8 & 8\\ & & 7 & 7 & 7 \\ & & & 4 & 4\\ \hline}$};
	\draw (2-0.265,2+0.65)--(2-0.265,2+1.35);
            \node at (3,3) (F) [label = {[xshift = -0.75cm] $U'$}]{$\tableau{& & & & \\ & \\ & & & & \\ & & & 8 & 8\\ & & 7 & 7 & 7 \\ & & & 4 & 4\\ \hline}$};
	\draw (3-0.265,2+0.65)--(3-0.265,2+1.35);
            \node at (4,3) (G) [label = {[xshift = -0.75cm] $V'$}]{$\tableau{& & & & \\ & & 7 & 8 & 8 \\ & & & & \\ & & &  & \\ & &  & 7 & 7 \\ & & & 4 & 4\\ \hline}$};
	\draw (4-0.265,2+0.65)--(4-0.265,2+1.35);
            \node at (1,2) (H) [label = {[xshift = -0.75cm] $V$}]{$\tableau{& 7 & & 9& 9\\ & & 7 & 8 & 8 \\ & & & & \\ & & &  & \\ & &  & 7 & 7 \\ & & & 4 & 4\\ \hline}$};
	\draw (1-0.265,1+0.65)--(1-0.265,1+1.35);
            \node at (2,2) (I) [label = {[xshift = -0.75cm] $W$}]{$\tableau{&  & & 9& 9\\ & 7 & 7 & 8 & 8 \\ & & & & \\ & & &  & \\ & &  & 7 & 7 \\ & & & 4 & 4\\ \hline}$};
	\draw (2-0.265,1+0.65)--(2-0.265,1+1.35);
            \node at (3,2) (J) [label = {[xshift = -0.75cm] $W'$}]{$\tableau{&  & & & \\ & 7 & 7 & 8 & 8 \\ & & & & \\ & & &  & \\ & &  & 7 & 7 \\ & & & 4 & 4\\ \hline}$};
	\draw (3-0.265,1+0.65)--(3-0.265,1+1.35);
            \node at (4,2) (K) [label = {[xshift = -0.75cm] $T_{\aa''}$}]{$\tableau{&  & & & \\ &  &  &  &  \\ & & & & \\ & & & 8 & 8\\ & 7&7  & 7 & 7 \\ & & & 4 & 4\\ \hline}$};
	\draw (4-0.265,1+0.65)--(4-0.265,1+1.35);
            \node at (2,1) (L) {$\tableau{&  & &9 &9 \\ &  &  &  &  \\ & & & & \\ & & & 8 & 8\\ & 7&7  & 7 & 7 \\ & & & 4 & 4\\ \hline}$};
	\draw (2-0.265,0.65)--(2-0.265,1.35);
            \node at (3,1) (M) [label = {[xshift = -0.75cm] $T_\aa$}]{$\tableau{&  & & & \\ &  &  &  &  \\ & & & 9 & 9 \\ & & & 8 & 8\\ & 7&7  & 7 & 7 \\ & & & 4 & 4\\ \hline}$};
	\draw (3-0.265,0.65)--(3-0.265,1.35);
            \draw[dashed,->] (A) -- (B) node[midway,above] {delete};
            \draw[dashed,->] (B) -- (C) node[midway,above] {$S_1$};
            \draw[thick,->] (A) -- (D) node[midway,right] {$S_1$};
            \draw[dashed,->] (D) -- (E) node[midway,above] {$e_5\circ e_4$};
            \draw[dashed,->] (E) -- (F) node[midway,above] {delete};
            \draw[dashed,->] (F) -- (G) node[midway,above] {$S_2$};
            \draw[thick,->] (D) -- (H) node[midway,right] {$S_2\circ e_5 \circ e_4$};
            \draw[thick,->] (H) -- (I) node[midway,above] {$f_5$};
            \draw[dashed,->] (I) -- (J) node[midway,above] {delete};
            \draw[dashed,->] (J) -- (K) node[midway,above] {$S_3$};
            \draw[thick,->] (I) -- (L) node[midway,right] {$S_3$};
            \draw[thick,->] (L) -- (M) node[midway,above] {$f_4^2 \circ f_5^2$};
        \end{tikzpicture}
    \end{center}
    \caption{An explicit example of the inductive argument in the proof of Theorem \ref{thm:lock-connected} with diagrams labeled. Each $S_i$ is a sequence of operators given by the inductive assumption, and the full sequence applied to $T$ to get to $T_\aa$ is given by following the southwest border.}
    \label{fig:proof-example}
\end{figure}

We will see in the next section that in addition to being connected, the lock crystal of $\aa$ forms a subcrystal of the key crystal of $\aa$.

%
\section{Rectification and Unlock}
%
\label{sec:rect-unlock}

The aim of this section is to prove the following theorem.

\begin{theorem} \label{thm:dem-subcrystal}
Let $\comp{a}$ be a weak composition. Then there is an injective, weight-preserving map $U_{\flat(\comp{a})}$ from $\LKT(\comp{a})$ to $\KKT(\comp{a})$ such that for any $T \in \LKT(\comp{a})$, if $e_i(T) \neq 0$, then $U_{\flat(\comp{a})}\circ e_i (T) = e_i \circ U_{\flat(\comp{a})}(T)$, and if $f_i(T) \neq 0$, then $U_{\flat(\comp{a})}\circ f_i (T) = f_i \circ U_{\flat(\comp{a})}(T)$. That is, the injection intertwines crystal operators on Kohnert tableaux.
\end{theorem}

\begin{corollary}
Given a weak composition $\comp{a}$, the lock crystal of $\comp{a}$ is a subcrystal of the key crystal of $\comp{a}$.
\end{corollary}

We will show this by comparing the rectification operators of Assaf and Gonz\'alez \cite{AG19} that act on Kohnert diagrams and new operators which we will call \emph{unlock operators} that act on labeled Kohnert diagrams. We will see that unlock operators on lock Kohnert tableaux act on the underlying diagram in the same way as rectification operators with the added benefit that unlock operators can track the movement of labels through each step. Once Theorem \ref{thm:dem-subcrystal} is proven, our desired result on the difference of a key and lock polynomial follows immediately.
We begin by defining rectification operators.

\begin{definition}[\cite{AG19}]
Given any diagram $D$ with $n\geq 1$ columns and integer $1 \leq i < n$, define the \emph{horizontal $i$-pairing} of $D$ as follows: $i$-pair any boxes in columns $i$ and $i+1$ that are located in the same row and then interatively $i$-pair any unpaired box in column $i+1$ with the lowest unpaired box in column $i$ located in a row above it whenever all the boxes in columns $i$ and $i+1$ in the rows between them are already horizontally $i$-paired.
\end{definition}

\begin{definition}[\cite{AG19}]
Given any integer $n\geq 0$ and any diagram $D$ with at most $n$ columns, for any integer $1\leq i < n$, define the \emph{rectification operator} $\ee_i$ on the space of diagrams as the operator which pushes the bottom-most horizontally unpaired box in column $i+1$ of $D$ left to column $i$. If $D$ has no unpaired boxes in column $i+1$, then $\ee_i(D) = 0$.
\end{definition}

As Assaf and Gonz\'alez note, these operators can be viewed as a rotation of raising operators on Kohnert diagrams. We also have the following equivalent formulation that originates from the characterization of key Kohnert diagrams in \cite{AS18b}(Lemma 2.2), which we find more convenient to work with in the proofs to follow. Given a diagram $D$ and an integer $i\geq 1$, define

\begin{equation}
    M^i(D,r) = \# \{(s,i+1)\in D | s \geq r \} - \# \{(s,i)\in D | s\geq r\},
\end{equation}
\begin{equation}
    M^i(D) = \underset{r}{\max}(M^i(D,r)).
\end{equation}

\begin{definition}
Given a positive integer $i \geq 1$, define the \emph{rectification operators} $\ee_i$ on Kohnert diagrams as follows. If $M^i(D) \leq 0$, then set $\ee_i(D)=0$; otherwise, letting $r$ be the largest row index such that $M^i(D,r) = M^i(D)$, set $\ee_i(D)$ to be the result of pushing the cell in position $(r,i+1)$ left to position $(r,i)$.
\end{definition}

We can see that this is equivalent because the largest row index on which $M^i(D,r)$ achieves its maximum is the same row as the lowest row containing a horizontally unpaired box in column $i+1$.
Now for $\comp{a}$ a weak composition, $m=\mathrm{max}(a_i)$, and $\alpha = \flat(\comp{a})$, let $R_{\alpha,i}$ denote the composition of rectification moves
\begin{equation}
   R_{\alpha,i} = (\ee_{\alpha_i}\circ \cdots \circ \ee_{m-1})\circ \cdots \circ (\ee_2\circ \cdots \circ \ee_{m-\alpha_i+1})\circ (\ee_{1}\circ \cdots \circ \ee_{m-\alpha_i}),
\end{equation}
and let $R_{\alpha}$ denote the composition of rectification moves
\begin{equation}
 R_\alpha =   R_{\alpha,\ell(\alpha)} \circ \cdots \circ R_{\alpha,2}\circ R_{\alpha,1}.
\end{equation}

We will sometimes refer to $R_\alpha$ as the \emph{Rectification algorithm} (for $\comp{a}$) and to each individual rectification operator that $R_\alpha$ is composed of as the \emph{steps} of the algorithm. We note that the order of rectification operators applied here is different in general from the order used by Assaf and Gonz\'alez.

\begin{figure}[ht]
	\begin{center}
	\begin{tikzpicture}[xscale=2.2,yscale=1.6]
		\node at (0,1) (T01) {$\tableau{ \times & \bullet & \bullet \\ & \bullet \\ \times & \bullet & \bullet \\ & & \bullet \\ & & \times\\ \hline}$};
	\draw (-0.24,0.45)--(-0.24,1.55);
		\node at (1,1) (T11) {$ \tableau{ \bullet & \bullet & \times \\ & \times \\ \bullet & \bullet & \times \\ & & \times \\ & \times\\ \hline}$};
	\draw (1-0.24,0.45)--(1-0.24,1.55);
		\node at (2,1) (T21) {$\tableau{ \bullet & \bullet & \times \\ & \times \\ \bullet & \bullet & \times \\ & & \times \\  \times\\ \hline}$};
	\draw (2-0.24,0.45)--(2-0.24,1.55);
		\node at (3,1) (T31) {$ \tableau{ \times & \bullet & \bullet \\  \times \\ \times & \bullet & \bullet \\ & & \times \\  \times\\ \hline}$};
	\draw (3-0.24,0.45)--(3-0.24,1.55);
		\node at (4,1) (T41) {$\tableau{ \times & \times & \times \\  \times \\ \times & \times & \times \\ & \times \\  \times\\ \hline}$};
	\draw (4-0.24,0.45)--(4-0.24,1.55);
		\draw[thick,->] (T01) -- (T11) node[midway,above] {$\ee_2$};
		\draw[thick,->] (T11) -- (T21) node[midway,above] {$\ee_1$};
		\draw[thick,->] (T21) -- (T31) node[midway,above] {$\ee_1$};
		\draw[thick,->] (T31) -- (T41) node[midway,above] {$\ee_2$};
		
	\end{tikzpicture}
	\caption{\label{fig:rect-103032} For $\aa = (1,0,3,0,3,2)$, we have $\alpha = (1,3,3,2)$ and $R_\alpha = \protect\ee_2 \protect\ee_1 \protect\ee_1 \protect\ee_2$. On the left is a lock Kohnert diagram of $\aa$ and each step of the Rectification algorithm for $\aa$ on that diagram with relevant horizontally paired boxes represented by bullets. }
	\end{center}

\end{figure}
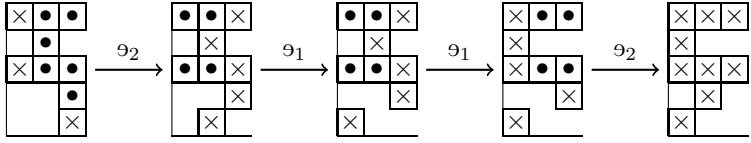

We have the following crucial properties of rectification operators which we will leverage in our proof of Theorem \ref{thm:dem-subcrystal}.

\begin{proposition}\label{prop:rect-weight-inject}
Given a weak composition $\aa$ and a Kohnert diagram $D$, if $R_\alpha(D) \neq 0$, then $R_\alpha$ is weight preserving and injective.
\end{proposition}

\begin{proof}
A rectification operator only pushes boxes to the left, so at every step, the number of boxes in each row remains unchanged. Injectivity follows from the observation that these can be viewed as rotated raising operators, and that raising operators are injective by definition. 
\end{proof}

\begin{theorem}[\cite{AG19},Theorem 5.33] \label{thm:rect-commute}
The rectification operators and the raising operators on diagrams commute. That is, given any diagram $D$ for which $\ee_c(D)\neq 0$, then for any row index $r\geq 1$, $e_r(D) \neq 0$ if and only if $e_r(\ee_c(D)) \neq 0$. Likewise, if $e_r(D) \neq 0$, then for any column index $c\geq 1$, $\ee_c(D) \neq 0$ if and only if $\ee_c(e_r(D)) \neq 0$. In this case, we have $\ee_c(e_r(D)) = e_r(\ee_c(D))$ for all values of $r$ and $c$ for which $\ee_c(D) \neq 0$ and $e_r(D) \neq 0$.
\end{theorem}

\begin{corollary}\label{cor:rect-nonzero}
Given a weak composition $\aa$, $\alpha = \flat(\aa)$, and $T\in \LKT(\aa)$, we have $R_\alpha(T) \neq 0$.
\end{corollary}

\begin{proof}
It is straightforward to check that $R_\alpha(T_\aa)$ is the unique key Kohnert tableau of content $\aa$ and weight $\alpha$, and therefore $R_\alpha(T_\aa) \neq 0$. Then since rectification operators intertwine crystal operators on Kohnert diagrams by Theorem \ref{thm:rect-commute} and the lock crystal is connected by Theorem \ref{thm:lock-connected}, we must have $R_\alpha(T) \neq 0$ as well. 
\end{proof}

\begin{definition}
Given a positive integer $i\geq 1$, define the \emph{unlock operators} $u_i$ on Kohnert tableaux as follows. The \emph{string} $\ell$ for $\ell$ a label of a labeled Kohnert tableau is the set of boxes of that tableau with label $\ell$. A box $x$ in string $\ell$ is \emph{left justified} if every column to the left of $x$ contains a box with label $\ell$.
We say that a box $x$ in a string $\ell$ in column $i+1$ \emph{crosses} a string $\ell'\neq \ell$ when string $\ell'$ contains a box in column $i$ weakly above $x$ and a box in column $i+1$ strictly below $x$.
Let $x$ be the box in column $i+1$ that has minimal label $\ell$ among those in column $i+1$ that are not left justified. If no such $x$ exists, then $u_i$ returns $0$. Otherwise, $u_i$ returns the diagram resulting from iterating the following steps until $x$ is pushed into column $i$.

\begin{enumerate}[label=\arabic*.]
\item If $x$ does not cross any strings, then push $x$ one space to the left and terminate the algorithm. Otherwise, go to step 2.
\item Fix string $\ell'$ to be the string with highest row index in column $i$ among those strings that $x$ crosses. Let $y$ be the box in column $i+1$ of string $\ell'$ and swap the row indices of $x$ and $y$ so that $x$ with label $\ell$ is below $y$ with label $\ell'$ in column $i+1$. Return to step 1.
\end{enumerate}
\end{definition}

We will see that this is guaranteed to terminate. Each time the steps loop, the row index of $x$ strictly decreases, so the only case in which the loop could get stuck is if $x$ is directly to the right of the rightmost box of some other string. In this case, $x$ would not cross that string, but it cannot be pushed to the left either. Claim (1) of Lemma \ref{lemma:key-conditions} ensures that this cannot happen.

For $\comp{a}$ a weak composition with $m=\max(a_i)$ and $\alpha = \flat(\comp{a})$, let $U_{\alpha,i}$ denote the composition of unlock operators
\begin{equation}
   U_{\alpha, i} =  (u_{\alpha_i}\circ \cdots \circ u_{m-1})\circ \cdots \circ (u_2\circ \cdots \circ u_{m-\alpha_i+1})\circ (u_{1}\circ \cdots \circ u_{m-\alpha_i}),
\end{equation}
and let $U_\alpha$ denote the composition of unlock operators
\begin{equation}
    U_\alpha =  U_{\alpha,\ell(\alpha)} \circ \cdots \circ U_{\alpha,2}\circ U_{\alpha,1}.
\end{equation}

As with $R_\alpha$, we will sometimes refer to $U_\alpha$ as the \emph{Unlock algorithm} (for $\comp{a}$) and to each individual unlock operator that it is composed of as the \emph{steps} of the algorithm. We note that in general, the unlock operators are not well defined for all labeled diagrams. However, we will show later in Lemma \ref{lemma:unlock-well-defined} that for any lock Kohnert tableau $T$ of shape $\comp{a}$, $U_\alpha(T)$ is well defined.

\begin{figure}[ht]
	\begin{center}
	\begin{tikzpicture}[xscale=2.2,yscale=1.6]
		\node at (0,1) (T01) {$\tableau{ 5 &6 & 6 \\ &5 \\ 3 & 3 & 5 \\ & & 3 \\ & &1\\ \hline}$};
	\draw (-0.24,0.45)--(-0.24,1.55);
		\node at (1,1) (T11) {$\tableau{ 5 &6 & 6 \\ &5 \\ 3 & 3 & 5 \\ & & 3 \\ & 1\\ \hline}$};
	\draw (1-0.24,0.45)--(1-0.24,1.55);
		\node at (2,1) (T21) {$\tableau{ 5 &6 & 6 \\ &5 \\ 3 & 3 & 5 \\ & & 3 \\  1\\ \hline}$};
	\draw (2-0.24,0.45)--(2-0.24,1.55);
		\node at (3,1) (T31) {$\tableau{ 5 &5 & 6 \\ 6 \\ 3 & 3 & 5 \\ & & 3 \\  1\\ \hline}$};
	\draw (3-0.24,0.45)--(3-0.24,1.55);
		\node at (4,1) (T41) {$\tableau{ 5 &5 & 5 \\ 6 \\ 3 & 3 & 3 \\ & 6 \\  1\\ \hline}$};
	\draw (4-0.24,0.45)--(4-0.24,1.55);
		\draw[thick,->] (T01) -- (T11) node[midway,above] {$u_2$};
		\draw[thick,->] (T11) -- (T21) node[midway,above] {$u_1$};
		\draw[thick,->] (T21) -- (T31) node[midway,above] {$u_1$};
		\draw[thick,->] (T31) -- (T41) node[midway,above] {$u_2$};
		
	\end{tikzpicture}
	\caption{\label{fig:unlock-103032} For $\comp{a} = (1,0,3,0,3,2)$, we have $\alpha = (1,3,3,2)$ and $U_\alpha = u_2 u_1 u_1 u_2$. On the left is a lock Kohnert tableau of content $\comp{a}$ and each step of the Unlock algorithm for $\comp{a}$ on that tableau. Compare with Figure \ref{fig:rect-103032}.}
	\end{center}

\end{figure}
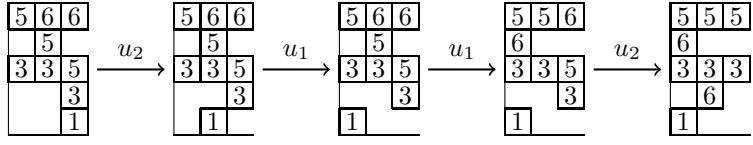

We also note that the order of the unlock operators in $U_\alpha$ is very intentionally chosen so that the boxes of $T \in \LKT(\aa)$ are left justified in a particular order. See Figure \ref{fig:unlock-103032} for a small example.

\begin{proposition}\label{prop:left-just-and-crossing}
Let $\aa$ be a weak composition with $\alpha = \flat(\aa)$ and with nonzero parts $\{a_{\ell_1},a_{\ell_2},\ldots, a_{\ell_k}\}$.
If $U_\alpha$ is well defined on $T \in \LKT(\comp{a})$, then in order from $i=1$ to $i=k$, $U_{\alpha,i}$ left justifies the boxes of string $\ell_i$ in order from left to right. Furthermore, at each step of the Unlock algorithm, a box $x$ with label $\ell$ can only cross strings with labels strictly smaller than $\ell$.
\end{proposition}

\begin{proof}
The first claim on the order of boxes moved by $U_\alpha$ can be seen by construction from a straightforward examination of the definition of unlock operators and lock Kohnert tableaux.

The second claim follows from the facts that strings weakly descend from left to right, unlock operators can only change the positions of boxes in the southwest direction of the box being left justified, and 
a box $x$ with label $\ell$ in $T$ is strictly lower than any box $y$ with label $\ell_j>\ell$ in the same column.

\end{proof}

It would be nice if each lock crystal had a unique lowest weight element. In this case, we would only need to show that this unique element maps to a key Kohnert tableau via Rectification, and then we could use the connectivity of the lock crystal and the commutativity of rectification operators and raising operators to prove Theorem \ref{thm:dem-subcrystal}. 
This is unfortunately not the case, and so instead we organize the proof of Theorem \ref{thm:dem-subcrystal} as follows. It is easier to first assume that step by step for a given lock Kohnert tableau $T$, $R_\alpha(\mathbb{D}(T))$ and $U_\alpha(T)$ \emph{agree on the level of diagrams}. That is, if we let $R_\alpha = \ee_{j_t}\circ\cdots\circ \ee_{j_1}$ and $U_\alpha = u_{j_t} \circ\cdots\circ u_{j_1}$, then for all $1 \leq s \leq t$, we have
$$\ee_{j_s} \circ \cdots\circ \ee_{j_1}(\mathbb{D}(T)) = \mathbb{D}(u_{j_s}\circ \cdots \circ u_{j_1}(T)).$$

Given this assumption, we show that the resulting diagram $U_\alpha(T)$ is a key Kohnert tableau of content $\comp{a}$ (a consequence of Lemma \ref{lemma:key-conditions}). We then show that the assumption always holds that $R_\alpha$ and $U_\alpha$ agree on the level of diagrams for lock Kohnert tableaux (a claim of Lemma \ref{lemma:unlock-well-defined}). 
We begin with the following technical results (Lemma \ref{lemma:e-truncation} and Corollary \ref{cor:l-truncation}).

In all the lemmas below, $T$ is a lock Kohnert tableau of content $\comp{a} = (a_1,\ldots,a_m)$ that contains the labels $\ell_1 < \cdots < \ell_k$, and $\alpha = \flat(\aa)$. We also define a \emph{truncation} of $T$, denoted $T^{<\ell}$, by deleting all boxes of $T$ with label $\ell$ or larger. From the definition of lock Kohnert tableaux, $T^{<\ell}$ is clearly still a lock Kohnert tableau.

\begin{figure}[ht]
\begin{center}
	\begin{tikzpicture}[xscale = 2, yscale = 2]
	\draw (-0.263,0.56)--(-0.263,1.44);
	\draw (1-0.263,0.56)--(1-0.263,1.44);
	\node at (0,1) (T00) {$ \tableau{ 5 &6 & 6 \\ &5 \\ 3 & 3 & 5 \\ & & 3 \\ & &1\\ \hline}$};
	\node at (1,1) (T10) {$\tableau{  & &  \\ & \\ 3 & 3 &  \\ & & 3 \\ & &1\\ \hline}$};
	\end{tikzpicture}
	\caption{\label{fig:truncation} On the left is a lock Kohnert tableau $T$ and on the right is $T^{<5}$.}
\end{center}

\end{figure}
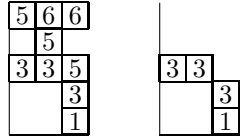

\begin{lemma}\label{lemma:e-truncation}
Fix $1 \leq p < k$ and $\ell > \ell_p$, and let $t$ be given by writing $R_{\alpha,p}\circ\cdots\circ R_{\alpha,1} = \ee_{p_t}\circ\cdots\circ \ee_{p_1}$. 
Then for all $1 \leq s < t$, $\ee_{p_s}$ pushes a box from position $(c+1,r)$ to $(c,r)$ in $\mathbb{D}(\ee_{p_{s-1}} \circ\cdots\circ \ee_{p_1}(T^{<\ell}))$ if and only if $\ee_{p_s}$ pushes a box from position $(c+1,r)$ to $(c,r)$ in $\mathbb{D}(\ee_{p_{s-1}}\circ\cdots\circ \ee_{p_1}(T))$.

\end{lemma}

\begin{proof}
Let $D$ be an arbitrary Kohnert diagram and let columns $c,c+1$ be such that column $c+1$ is nonempty. Furthermore, let $h_c,h_{c+1}$ be the highest row indices occupied by boxes in columns $c,c+1$ respectively, where $h_c = 0$ if column $c$ is empty. Suppose that $M^c(D) > 0$ and $r_0$ is the highest row index for which $M^c(D,r)$ achieves its maximum. Then the following hold from the definition of $M^c$ by examining $M^c(D_i,r)$ compared to $M^c(D,r)$ in each case over all rows.

\begin{enumerate}
    \item Let $r_{c+1} > h_{c+1}$ and let $D_1$ be obtained from $D$ by adding a box to position $(c+1, r_{c+1})$. Then $r_0$ is the highest row index for which $M^c(D_1,r)$ achieves its maximum.
    \item Let $r_c \geq r_{c+1}$ with $r_c > h_c$ and $r_{c+1} > h_{c+1}$. Obtain $D_2$ from $D$ by adding boxes to positions $(c,r_c)$ and $(c+1,r_{c+1})$. Then $r_0$ is the highest row index for which $M^c(D_2,r)$ acheives its maximum.
    \item Suppose that $ h_{c+1} > r_0$ and obtain $D_3$ from $D$ by removing the box in position $(c+1, h_{c+1})$. If $M^c(D_3)>0$, then $r_0$ is the highest index for which $M^c(D_3,r)$ achieves its maximum.
    \item Suppose $h_c \geq h_{c+1} > r_0$, and obtain $D_4$ from $D$ by removing the boxes in positions $(c,h_c)$ and $(c+1, h_{c+1})$. If $M^c(D_4)> 0$, then $r_0$ is the highest index for which $M^c(D_4,r)$ achieves its maximum.
\end{enumerate}

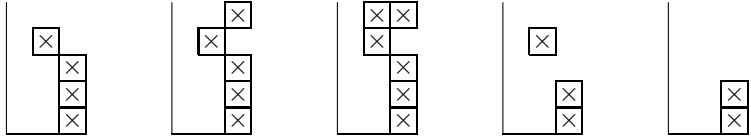
\begin{figure}[ht]
	\begin{center}
	\begin{tikzpicture}[xscale=2.2,yscale=1.6]
\draw (-0.24,0.447)--(-0.24,1.55);
\draw (1-0.24,0.447)--(1-0.24,1.55);
\draw (2-0.24,0.447)--(2-0.24,1.55);
\draw (3-0.24,0.447)--(3-0.24,1.55);
\draw (4-0.24,0.447)--(4-0.24,1.55);
	\node at (0,1) (T00) {$ \tableau{& & \\& \times \\& & \times \\& & \times \\& & \times\\ \hline}$};
	\node at (1,1) (T10) {$ \tableau{ &&\times  \\ &\times \\ && \times \\ && \times \\ && \times\\ \hline}$};
	\node at (2,1) (T20) {$ \tableau{& \times &\times  \\& \times \\ && \times \\ && \times \\ && \times\\ \hline}$};
	\node at (3,1) (T30) {$ \tableau{ &  \\ &\times \\ &  \\ && \times \\& & \times\\ \hline}$};
	\node at (4,1) (T40) {$\tableau{ &  \\ &  \\ &  \\ && \times \\ && \times\\ \hline}$};
	\end{tikzpicture}
	\caption{\label{fig:four-diagrams-example} In order from left to right, we have an example of a possible diagram $D$ and diagrams $D_1$ through $D_4$. In all cases, $E^2$ pushes the box in position $(1,3)$ to $(1,2)$.}
	\end{center}
\end{figure}

See Figure \ref{fig:four-diagrams-example} for an example of each case. By the definition of lock Kohnert tableaux, going from $T^{<\ell}$ to $T$ by adding back strings one at a time either has no effect on a pair of columns $c,c+1$ or it has the effect of one of the cases $(1)$ or $(2)$ above, which proves one direction of the claim.

Similarly, removing strings one at a time from $T$ to obtain $T^{<\ell}$ either has no effect on a pair of columns $c,c+1$ or it has the effect of one of the cases $(3)$ or $(4)$ above. We do need to check that it is still true that the $M^c(D_3)>0$ and $M^c(D_4)>0$ conditions hold in cases $(3)$ and $(4)$ respectively. Using Corollary \ref{cor:rect-nonzero}, we see that $R_{\alpha',p}\circ\cdots\circ R_{\alpha',1}$ is nonzero on $T^{<\ell}$, which must mean that $M^c(D_i)>0$ does hold for cases (3) and (4).

\end{proof}

Using the same notation as above, we obtain the following corollary.

\begin{corollary}\label{cor:l-truncation}
Suppose that for every $1\leq s \leq t$, $u_{p_s}\circ\cdots\circ u_{p_1}(T)$ is well defined and we have
$$\ee_{p_s}\circ\cdots\circ \ee_{p_1}(\mathbb{D}(T)) = \mathbb{D}(u_{p_s} \circ\cdots\circ u_{p_1}(T)).$$
Then for all $1\leq s \leq t$, $u_{p_s} \circ\cdots\circ u_{p_1}(T^{<\ell_q})$ is well-defined and we have
$$\ee_{p_s}\circ\cdots\circ \ee_{p_1}(\mathbb{D}(T^{<\ell_q})) = \mathbb{D}(u_{p_s} \circ\cdots\circ u_{p_1}(T^{<\ell_q})).$$

\end{corollary}

\begin{proof}

Proposition 1.5 tells us that the operators $U_{\alpha,p}\circ\cdots\circ U_{\alpha,1}$ left justify the boxes in strings $\ell_1$ through $\ell_p$. By the weakly decreasing row conditions and strictly decreasing column conditions on lock Kohnert tableaux as well as the fact that unlock operators only push boxes southwest, the left justification of the strings $\ell_1,\ldots,\ell_p$ can only depend on the positions of boxes in those strings. Therefore, removing any string $\ell_m > \ell_p$ from $T$ has no effect on the steps of the Unlock algorithm up through the left justification of string $\ell_p$. It follows that $u_{p_s}\circ\cdots\circ u_{p_1}(T^{<\ell_q})$ and $u_{p_s}\circ\cdots\circ u_{p_1}(T)$ have identical strings $\ell_1,\ldots,\ell_p$ for all $1\leq s \leq t$, then
combining with Lemma \ref{lemma:e-truncation} proves the claim.

\end{proof}

We will use this corollary in proving the following lemma that if $U_\alpha$ and $R_\alpha$ agree on the level of diagrams on lock Kohnert tableaux,
then the Unlock algorithm preserves the properties necessary for the resulting tableau to be a key Kohnert tableau of the same content as the inputted lock Kohnert tableau. In particular, compare claims (2), (3), and (4) to the definition of key Kohnert tableaux.

\begin{lemma} \label{lemma:key-conditions}
Write 
$$R_\alpha = R_{\alpha,k}\circ\cdots\circ R_{\alpha,1} = \ee_{k_t}\circ\cdots\circ \ee_{k_1}$$
$$U_\alpha = U_{\alpha,k}\circ\cdots\circ U_{\alpha,1} = u_{k_t} \circ \cdots\circ u_{k_1}$$
and suppose that for $1\leq s \leq t$, $u_{k_s}\circ\cdots\circ u_{k_1}(T)$ is well defined and we have 
$$\ee_{k_s}\circ\cdots\circ \ee_{k_1}(\mathbb{D}(T)) = \mathbb{D}(u_{k_s}\circ\cdots\circ u_{k_1}(T))$$
Then the following hold:
\begin{enumerate}
    \item An operator $u_{k_i}$, $1\leq i \leq s$, never tries to push a box $x$ from $(c+1,r)$ to $(c,r)$ where column $c$ contains the rightmost box of a different string in some row weakly above $r$.
    \item After all steps of $U_{\alpha,i}$ have been completed, the string $\ell_i$ is left justified and weakly descending in row index from left to right and remains so through every subsequent step of $U_\alpha$. Furthermore, while the steps of $U_{\alpha,i}$ are in progress, all other strings than $\ell_i$ maintain their weakly decreasing property.
    \item (inversions) For each intermediate labeled diagram $u_{k_i}\circ\cdots\circ u_{k_1}(T)$ with $1 \leq i \leq t$, if a column $c$ has boxes $x,y$ where $x$ is both below $y$ and has a larger label, then in column $c+1$, there is a box $z$ strictly above the row index of $x$ with the same label as $y$. 
    \item (flagged) For each intermediate diagram $u_{k_i}\circ\cdots\circ u_{k_1}(T)$ with $1 \leq i \leq t$, every box with label $\ell$ is no higher than row $\ell$.
\end{enumerate}

\end{lemma}

\begin{proof}

We proceed by induction on the strings of $T$ in increasing value of label. For the base case, the claims of Proposition \ref{prop:left-just-and-crossing} make it straightforward to check that while applying $U_{\alpha,1}$, the claims hold at every step. Now suppose that for $1 < p \leq k$, the claims hold through all steps of $U_{\alpha,1},U_{\alpha,2},\ldots,U_{\alpha,p-1}$.

\textbf{Proof of claim (1).} Suppose that all conditions hold up to some $u_{k_i}$, and let $u_{k_{i-1}}\circ\cdots\circ u_{k_1}(T) = T'$. Suppose that $u_{k_i}$ chooses a box $x$ with label $\ell$. Conduct all swaps of $x$ that occur in step 2 while applying $u_{k_i}$ to $T'$, but stop just before $u_{k_i}$ tries to push $x$ left after all swaps have occurred. Let $y$ be the rightmost box of some other string that is in column $k_i$ and weakly above $x$. At this point, the underlying diagram is unchanged, so if we delete all boxes with labels larger than $x$ to get $T^{\prime<\ell}$, then by Lemma \ref{lemma:e-truncation}, $R_\alpha(\mathbb{D}(T)) \neq 0$ means that $\ee_{k_i}(\mathbb{D}(T^{\prime<\ell})) \neq 0$, so $M^{k_i}(\mathbb{D}(T^{\prime<\ell}))>0$. 

Since $x$ is weakly below $y$, deleting both $x$ and $y$ from $T^{\prime<\ell}$ to get $T^{\prime\prime<\ell}$ preserves $M^{k_i}(\mathbb{D}(T^{\prime\prime<\ell}))>0$. Since all smaller labeled strings are left justified, columns $k_i,k_i+1$ of $T^{\prime\prime<\ell}$ can either contain the rightmost box of a string or one box in each column from a string. Therefore, since smaller labeled strings are also weakly decreasing from left ot right, each string must contribute either $0$ or $-1$ to a given row, and so it must hold that $M^{k_i}(\mathbb{D}(T^{\prime\prime<\ell}))\leq 0$, which is a contradiction. Therefore, condition (1) must hold.

\textbf{Proof of claim (2).} We observe that if all strings with labels smaller than $\ell_p$ are weakly decreasing, then by definition, any swaps that occur during step 2 of an unlock operator between a box of string $\ell_p$ and a string $\ell_s < \ell_p$ will preserve the weakly descending property of string $\ell_s$. Proposition \ref{prop:left-just-and-crossing} tells us that an unlock operator trying to push a box with label $\ell_p$ left cannot change the position of any boxes in a string $\ell_t > \ell_p$. Therefore, strings $\ell_s > \ell_p$ remain weakly descending because they are in the original tableau $T$.

It remains to check that string $\ell_p$ is weakly descending from left to right after the steps of $U_{\alpha,p}$ are completed. Index the boxes of string $\ell_p$ from left to right as $x_1,\ldots, x_t$. Suppose that for all $x_j$ for $1< j \leq i < t$ it holds that $x_j$ is weakly lower than $x_{j-1}$ after they have been left justified, where the base case for $x_1$ is vacuously true. Suppose also that over the course of being left justified, $x_i$ was swapped $m$ times from positions $(c_1,r_0),\ldots,(c_m,r_{m-1})$ to $(c_1,r_1),\ldots,(c_m,r_m)$ respectively, with $r_0 > r_1 > \cdots > r_m$ and $c_1 \geq c_2 \geq \cdots \geq c_m$, and index the respective strings that $x_i$ swaps with as $\ell_{i_1},\ldots, \ell_{i_m}$.

Since string $\ell_p$ is weakly decreasing to begin with, $x_{i+1}$ must start in some row $r_0'\leq r_0$. We know that string $\ell_{i_1}$ has a box in position $(c_1,r_0)$, since $x_i$ swapped from position $(c_1,r_0)$ to $(c_1,r_1)$. By condition (1), $x_{i+1}$ cannot end up in the same column and strictly lower than a box in string $\ell_{i_1}$ unless there is some column $c_1' > c_1$ in which $x_{i+1}$ either swaps with string $\ell_{i_1}$ or swaps with some other string such that it ends up below some box of string $\ell_{i_1}$ in column $c_1'$. In either case, since string $\ell_{i_1}$ was already weakly decreasing before $x_i$ swapped with it in column $c_1$, its box in column $c_1+1$ must have a row index weakly less than $r_1$, and so by the time $x_{i+1}$ is pushed into column $c_1$, it must have a row index $r_1' \leq r_1$. 
If $r_1' \leq r_k$, then we are done. If we suppose instead that  $r_j \geq r_1' > r_{j+1}$ for some $1 \leq j < k$, then we can repeat the above argument with string $\ell_{i_{j+1}}$ to show that $x_{i+1}$ must end up in some row $r_2'\leq r_{j+1}$ before it reaches column $c_{j+1}$. Iterating this eventually forces $x_{i+1}$ to end up weakly below row $r_k$, and therefore weakly below $x_i$. Since $i$ was arbitrary, the entire string $\ell_p$ must be weakly decreasing left to right.

\textbf{Proof of claim (3).} By Proposition \ref{prop:left-just-and-crossing}, no strings $\ell_s> \ell_p$ have inversions at any step of $U_{\alpha,1},\ldots,U_{\alpha,p}$. 

If an unlock operator swaps a box $x$ of $\ell_p$ so that it is below the string $\ell_t<\ell_p$ in the same column, the operator terminates with a left push, so combined with the weakly decreasing property of string $\ell_t$, $x$ satisfies the inversion condition with the boxes of string $\ell_t$ directly after that operator is applied. Each successive unlock operator that left justifies $x$ moves it left or down, so condition (2) ensures that $x$ continues to satisfy the inversion condition with string $\ell_t$. Otherwise, $x$ stays above string $\ell_t$, and the inversion condition is also satisfied.

It remains to show that, given an intermediate diagram in which inversion conditions are satisfied everywhere at all previous steps, any subsequent swaps that occur in $U_{\alpha,p}$ do not violate inversion conditions between pairs of strings $\ell_s,\ell_t < \ell_p$. To do this, we consider the following two diagrams (with other boxes suppressed).

	
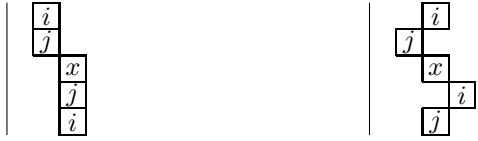
\begin{center}
	\begin{tikzpicture}[xscale=5,yscale=1.6]
\draw (-0.105,0.448)--(-0.105,1.548);
\draw (1-0.141,0.448)--(1-0.141,1.548);
		\node at (0,1) (T00) { $\tableau{ &  i \\ & j \\ & & x \\ & & j \\ & & i}$};
		\node at (1,1) (T10) {$\tableau{ & & i \\ & j\\ & & x \\ & & & i\\ & & j}$};
	\end{tikzpicture}
	\end{center}

We claim that if $x$ in row $r_1$ swaps with a box $y$ in row $r_2$, then any labels that appear between $x$ and $y$ have a smaller label than $y$. The diagram on the left gives an example of how there might be a larger label between $x$ and $y$. However, if $j>i$, then the inversion condition is violated between the boxes in positions $(2,4), (2,5),(3,1)$, which contradicts that our given diagram satisfies inversion conditions. 
 The right diagram shows the only way a swap might cause a trio of boxes that violates the inversion condition, with $j>i$, where a box from string $j$ remains below the box of string $i$ in the same column, but is moved weakly above a box of string $i$ in the next column to the right.
 
 The crux is how $x$ made it to that position. 
 If it was pushed left into that position, then it failed to swap with string $i$, so that cannot be possible. It could also have swapped with string $i$ into that position, but then prior to that swap, the $i$ in position $(3,3)$, the $j$ in position $(2,4)$, and the $i$ in column $2$ above the $j$ would violate the inversion condition.
 The last option is if $x$ swapped with some box $z$ with label $k$. However, by our previous claim, $k>i$, and then prior to $x$ and $z$ swapping, the inversion condition is not satisfied with $z$ in the position of $x$, which is again a contradiction.

\textbf{Proof of claim (4).} Proposition \ref{prop:left-just-and-crossing} shows that no string $\ell_s>\ell_p$ is changed while any string $\ell_1,\ldots,\ell_p$ is left justified, so boxes of string $\ell_s$ continue to satisfy the flagged condition because they did to begin with in $T$. Boxes of string $\ell_p$ can only move south or west while $U_{\alpha,p}$ is applied, so they must also continue to satisfy the flagged condition. Finally, the leftmost box of any string $\ell_t<\ell_p$ satisfies the flagged condition before $U_{\alpha,p}$ is applied by the inductive assumption. Unlock operators cannot change the position of the leftmost boxes of left justified strings, and such strings remain weakly decreasing from left to right by condition (2), so all boxes of strings $\ell_t<\ell_p$ must also satisfy the flagged condition through all steps of $U_{\alpha,p}$.

\end{proof}

Up to this point, we have been examining the consequences of the assumption that the Unlock algorithm is well defined on lock Kohnert tableaux and that it agrees with Rectification on the level of diagrams. We now show that this assumption indeed holds in general on lock Kohnert tableaux.

\begin{lemma}\label{lemma:unlock-well-defined}
Write $U_\alpha = u_{k_t}\circ\cdots\circ u_{k_1}$ and $R_\alpha = \ee_{k_t} \circ\cdots\circ \ee_{k_1}$. The function $U_\alpha$ is well defined and
$$\ee_{k_s} \circ\cdots\circ \ee_{k_1}(\mathbb{D}(T)) = \mathbb{D}(u_{k_s}\circ\cdots\circ u_{k_1}(T))$$
holds for all $1\leq s \leq t$.

\end{lemma}

\begin{proof}
We proceed by induction, noting that the following argument proves both the base case at $m=1$ and the inductive steps for $m>1$. Suppose that for some $m$, we have
$$\ee_{k_{s}} \circ\cdots\circ \ee_{k_1}(\mathbb{D}(T)) = \mathbb{D}(u_{k_{s}}\circ\cdots\circ u_{k_1}(T))$$
for all $1\leq s \leq m-1$, where $\ee_{k_{m-1}}\circ\cdots \circ \ee_{k_1}$ and $u_{k_{m-1}}\circ\cdots\circ u_{k_1}$ are the identity at $m=1$. We first show that $u_{k_m}$ is well defined on $u_{k_{m-1}}\circ\cdots\circ u_{k_1}(T) = T'$.

Proposition \ref{prop:left-just-and-crossing} shows that by construction, $u_{k_m}$ must have a box that it tries to push left, so the only way that it can not be well defined is if the box it attempts to push left is in a position where it is directly to the right of and in the same row as the rightmost box of a different string. In this case, there is nothing to swap with, but it still cannot be pushed left into an open space.
The proof of condition (1) of Lemma \ref{lemma:key-conditions} can be repeated here to show that this cannot happen (noting that the proof of condition (1) does not require the assumption that $u_{k_m}$ and $\ee_{k_m}$ agree on the level of diagrams), and therefore $u_{k_m}$ must be well defined on $T'$.

Now we check that 
$$\ee_{k_m}\circ\cdots\circ \ee_{k_1}(\mathbb{D}(T)) = \mathbb{D}(u_{k_m}\circ\cdots\circ u_{k_m}(T)).$$
Suppose that $u_{k_m}$ chooses a box $x$ to push left, with label $\ell$. 
Due to the weakly descending arrangement of labels in columns $k_m,k_m+1$ of $T^{\prime<\ell+1}$ as discussed above in the proof of condition (1) of Lemma \ref{lemma:key-conditions}, $\mathbb{D}(T'^{<\ell+1})$ has at most one horizontally unpaired box in column $k_{m+1}$, and it follows that we can at most have $M^{k_m}(T^{\prime<\ell+1}) = 1$, and if that maximum is acheived, it must be in the row containing the horizontally unpaired box.

Using Lemma \ref{lemma:e-truncation} and $\ee_{k_m}(T') \neq 0$, we know this maximum must be achieved somewhere. Let $r_0$ be the row of $x$ in $T'^{<\ell+1}$, and suppose $u_{k_m}$ swaps it to rows $r_1, r_2,\ldots,r_t$ before being pushed left. 
The descending arrangement of labels in columns $k_m,k_m+1$ means that a first upper bound for $r_{\max}$, the maximal row index such that $M^{k_m}(T^{\prime<\ell+1},r_{\max})=1$, is $r_0$. However, since $x$ swaps into row $r_1$, it must cross some string $\ell_{i_1}$ that has boxes at $(k_m,r_1')$ and $(k_m,r_1)$ with $r_1 < r_0 \leq r_1'$. Again using the descending arrangement of other labels, the string of $x$ is the only string that can cumulatively contribute $+1$ to $M^{k_m}(T^{\prime<\ell+1}, r)$, so since string $\ell_{i_1}$ cumulatively contributes $-1$ to $M^{k_m}(T^{\prime<\ell+1}, r)$ for all $r_1 < r \leq r_1'$, we must have $M^{k_m}(T^{\prime<\ell+1},r_0)\leq 0$. Therefore $r_1 < r_0$ gives a new upper bound on $r_{\max}$. Iterating this argument eventually gives an upper bound of $r_t$.

Now $x$ is in row $r_t$ and is not crossing any strings. Once again following the proof of condition (1) of Lemma \ref{lemma:key-conditions}, we get that all labels in columns $k_m,k_m+1$ that are above $x$ must have a box in both columns. Therefore, $M^{k_m}(T^{\prime<\ell+1},r_t)=1$ so the upper bound is achieved and $r_{\max} = r_t$. Then we have
$$\ee_{k_s}\circ\cdots\circ \ee_{k_1}(\mathbb{D}(T)) = \mathbb{D}(u_{k_s}\circ\cdots\circ u_{k_1}(T)),$$
which completes the proof of the inductive step.

\end{proof}

Combining Lemmas \ref{lemma:key-conditions} and \ref{lemma:unlock-well-defined} shows that the final diagram after applying the Unlock algorithm to a lock Kohnert tableau is a key Kohnert tableau of the same content and that the underlying diagram is the same as the one resulting from rectification. The rectification operators are weight-preserving, injective, and intertwine crystal operators on Kohnert diagrams, so Theorem \ref{thm:dem-subcrystal} follows, and we immediately obtain our result on the difference of a key and a lock polynomial.

\begin{corollary}
For $\aa$ a weak composition, the difference $\key_\aa - \lock_\aa$ is monomial positive.
\end{corollary}

\begin{proof}
Since Unlock provides an injective, weight-preserving map from lock Kohnert tableau to key Kohnert tableau, $\key_\aa - \lock_\aa$ is the generating polynomial of keys with the image of Unlock removed.
\end{proof}

\section*{Acknowledgements}

I am grateful to Sami Assaf for pointing me to this question and for the enlightening (and patient) conversations that followed and to Jim Haglund, Jongwon Kim, and Vasu Tewari for their support. The author was partially supported by the NSF Graduate Research Fellowship, DGE-1845298.

\bibliographystyle{amsalpha} 
\bibliography{locks}
\end{document}